\documentclass[11pt, reqno, english]{amsart}

\usepackage{enumerate,amsthm,bbm}
\usepackage{amssymb,amscd, amsmath, babel, amstext, amsthm, latexsym}

\newcommand*{\abs}[1]{\left|#1\right|}
\newcommand*{\set}[1]{\left\{#1\right\}}
\newcommand*{\norm}[1]{\left\Vert#1\right\Vert}

\renewcommand{\P}{\mathbf P}
\newcommand{\E}{\mathbf E}
\newcommand{\R}{\mathbb R}
\newcommand{\D}{\mathcal D}
\newcommand{\F}{\mathcal F}
\newcommand{\ind}{\mathbbm 1}
\newcommand{\B}{\mathcal B}
\newcommand{\plim}{\P\mkern2mu\text{-}\!\lim}
\DeclareMathOperator{\diam}{diam}

\newtheorem{theorem}{Theorem}[section]
\newtheorem{lemma}{Lemma}[section]
\theoremstyle{proposition}
\newtheorem{proposition}{Proposition}[section]
\newtheorem{corollary}{Corollary}[section]
\newtheorem{example}{Example}[section]
\theoremstyle{remark}
\newtheorem{remark}{Remark}[section]
\theoremstyle{definition}
\newtheorem{definition}{Definition}[section]

\allowdisplaybreaks[1]

\begin{document}
\title[Integration for Volterra processes]
{Fractional calculus and path-wise integration for Volterra processes driven by L\'evy and martingale noise}

\author{G. Di Nunno}\address{G. Di Nunno \\Department of Mathematics, University of Oslo, P.O. Box 1053 Blindern, N-0316 Oslo Norway. Email: giulian@math.uio.no }\address{G. Di Nunno\\ Norwegian School of Economics and Business Administration (NHH), Helleveien 30, N-5045 Bergen, Norway.}
\author{Y. Mishura}\address{Y.Mishura \\ Department of Probability Theory, Statistics and Actuarial Mathematics, Taras Shevchenko National University of Kyiv, 64 Volodymyrska, 01601 Kyiv, Ukraine. E-mail: myus@univ.kiev.ua } \author{K. Ralchenko}\address{K. Ralchenko \\ Department of Probability Theory, Statistics and Actuarial Mathematics, Taras Shevchenko National University of Kyiv, 64 Volodymyrska, 01601 Kyiv, Ukraine. Email: k.ralchenko@gmail.com} 


\date{August 30, 2016}
\maketitle

\vspace{-8mm}
\begin{abstract}
We introduce a pathwise integration for Volterra processes driven by L\'evy noise or martingale noise. These processes are widely used in applications to turbulence, signal processes, biology, and in environmental finance. Indeed they constitute a very flexible class of models, which include fractional Brownian and L\'evy motions and it is part of the so-called ambit fields.
A pathwise integration with respect of such Volterra processes aims at producing a framework where modelling is easily understandable from an information perspective.
 The techniques used are based on fractional calculus and in this there is a bridging of the stochastic and deterministic techniques. The present paper aims at setting the basis for a framework in which further computational rules can be devised. Our results are general in the choice of driving noise. Additionally we propose some further details in the relevant context subordinated Wiener processes.


\end{abstract}

\section{Introduction}

\setcounter{section}{1}
\setcounter{equation}{0}\setcounter{theorem}{0}

In this paper, we consider Volterra processes, namely processes of the form
\begin{equation}\label{eq:volt-proc}
Y_t=\int_0^t g(t,s)\,dZ_s,
\quad t\in[0,T],
\end{equation}
where $g(t,s)$ is a given deterministic Volterra-type kernel, and
$Z$ is a  L\'evy process or a (square integrable) martingale process.
The integral in \eqref{eq:volt-proc} is understood in the sense of \cite{RR} as taking the limit in probability of elementary integrals.
Volterra processes of the type above are widely used in physics for the modelling of turbulence, see e.g. \cite{BN_Schmiegel}, \cite{Emil}. Also they have been suggested in the context of biology/medicine for the modeling of cancer growth in biological tissues, see \cite{cancer}.
Furthermore, these processes have been used successfully in mathematical finance, specifically in energy finance where the spot prices of electricity and other commodities strongly depend on environmental risk factor, such as temperatures, wind speed, sun coverage, precipitations, etc.  Such processes also appear in problems of credit risk and are well suited to fit stochastic volatility models. See e.g. \cite{BNBV13, Benthbook1, Bishwal,Fink_13,Kluppelberg_15,Tikanmaki}, and reference therein.
Finally, Volterra processes \eqref{eq:volt-proc} also provide suitable models in signal processing, see e.g. \cite{Unser_14}, and
for the workload of network devices, see e.g. \cite{Wolpert_05}.

The Volterra processes \eqref{eq:volt-proc} are part of the general class of ambit fields, which appear within a space-time framework, while here we have only time, thus a process and the the integrand sees not only a deterministic kernel, but also a stochastic component. Such stochastic component can be also replaced by a time change in the driving noise, as it actually done in the present paper in terms of subordination. In the setting of ambit processes the so-called ambit set is here reduced to the real semi-line. See e.g. \cite{Podolskij} for a survey on ambit fields.
The class of processes defined in \eqref{eq:volt-proc} contains the fractional Brownian motion and its generalisation, namely, the fractional L\'evy process.
In fact, assume that the function $g$ is the Molchan-Golosov kernel, which  is given by
\[
g_H(t,s)=C_H
(t-s)^{H-\frac12}F\left(\tfrac12-H,H-\tfrac12,H+\tfrac12,\tfrac{s-t}{s}\right),
\quad0<s<t<\infty,
\]
and $g_H(t,s)=0$ otherwise, where $H\in(0,1)$,
\[
C_H=\frac1{\Gamma(H+\frac12)}\left(\frac{2H\Gamma(H+\frac12)\Gamma(\frac32-H)}{\Gamma(2-2H)}\right)^{\frac12},
\]
and $F$ is the Gauss' hypergeometric function.
If the driving process $Z$ is a Brownian motion, then the process $Y$ defined by~\eqref{eq:volt-proc} with the kernel $g_H$ is the fractional Brownian motion, see \cite{NVV}.
If $Z$ is a L\'evy process without Gaussian component such that $\E Z_1=0$ and $\E Z_1^2<\infty$, then $Y$ is
the fractional L\'evy process by Molchan-Golosov transformation (fLpMG), introduced in \cite{Tikanmaki} (see also \cite{Fink_13} for multivariate generalization).
Let us mention that there exist another definitions of fractional L\'evy processes in the literature. In particular, fractional L\'evy process by Mandelbrot-van Ness representation (fLpMvN) was defined in \cite{Benassi04} and studied in \cite{Marquardt06}. The comparison between fLpMvN and fLpMG can be found in \cite{Tikanmaki}.

The aim of the present paper is to develop a theory of integration with respect to processes of the form \eqref{eq:volt-proc} applying fractional calculus, thus generalising the famous construction for integrals w.r.t. fractional Brownian motion from \cite{Zahle98, Zahle99}.
The use of fractional calculus allows for a bridging between stochastic and deterministic methods, which is very interesting from the use of models.
Indeed our aim is to set the basis for a framework of pathwise calculus for Volterra processes. At present we concentrate on the definition and characterisation of the integrators and the integrands.
Future research will focus on the actual calculus rules.
It is important to have a manageable calculus from the applied perspective in which beyond the prediction on a model, other questions naturally appear linked e.g. to stochastic control.
In this paper we concentrate on the case when the driving noise $Z$ is a L\'evy process and when it is possible we will consider $Z$ to be a square integrable martingale. Also, we detail our results in the case in which $Z$ is a subordinate Brownian motion. Indeed subordination is one of the easy way to construct a L\'evy process having also advantages from the simulation point of view. See e.g. \cite{CT}.
In particular, processes with compound Poisson, stable and Gamma subordinators are studied in detail.

Several approaches to the integration with respect to L\'evy-driven Volterra processes are known.
In \cite{Bender_Marquardt}, a Skorokhod type integral was considered. That construction followed $S$-transform approach, developed in~\cite{Bender03} for fractional Brownian motion.
Another approach was proposed in \cite{BNBPV14} and then extended in \cite{DV}, where the integration operator was based on Malliavin calculus and described an anticipative integral.
Wiener integration with respect to fLpMG was considered in \cite{Tikanmaki}. But it turns out that one of the simplest and natural methods to construct the integral w.r.t. L\'evy-driven Volterra processes is to apply fractional calculus. This has the advantage that combines deterministic and stochastic techniques and it has a clear relationship with the underlying noise information flow.
Hence, in the present paper, following \cite{Zahle98}, we construct pathwise stochastic integral using fractional integrals and derivatives.
We present general conditions for the existence of this integral in terms of fractional derivatives. As an example we consider the case of fLpMG.

The paper is organised as follows.
In Section~\ref{sec:int} we review the construction of the integral of a deterministic kernel with respect to a L\'evy noise and a square integrable martingale. In particular we detail to case of the subordinated Wiener process in Section~\ref{sec:SWP}. The elements of fractional calculus are presented in Section \ref{sec:integration}. Finally, in Section~\ref{integration- central} is devoted to pathwise integrals with respect to Volterra processes. Various examples are provided at all stages.

\section{Integration with respect to L\'evy processes}\label{sec:int}

\setcounter{section}{2}
\setcounter{equation}{0}\setcounter{theorem}{0}

In this section we study the stochastic integrals with respect to L\'evy processes and square integrable martingales. We review the basic construction and we provide some results on the upper-bounds for the moments of the resulting integrals. Indeed these a priori bounds for the moments of order $p\geq 1$ are fundamental results for the development in the sequel.
Being very simple and also being a partial case of Bichteler-Jacod inequalities, these bounds with the values of corresponding constants containing the integrals w.r.t.  the L\'{e}vy measures,  are rather elegant therefore we provide the corresponding proofs.
We start by the definition of the integral using the approach of Rajput and Rosinski~\cite{RR}.

\subsection{Integration of non-random functions with respect to L\'evy process }\label{sec:int-1}
Let $Z=\set{Z_t,t\ge0}$  be a L\'evy process. Define
\[
\tau(z):=
\begin{cases}
z, & \abs{z}\le1,\\
\frac{z}{\abs{z}}, & \abs{z}>1.
\end{cases}
\]
Then the characteristic function of $Z_t$ can be represented in the following form (see, e.g., \cite{sato})
\[
\E\exp\set{i\mu Z_t}
=\exp\set{t\Psi(\mu)},
\]
where \[
\Psi(\mu)=ib\mu-\frac{a\mu^2}{2}
+\int_{\R}\left(e^{i\mu x}-1-i\mu\tau(x)\right)\pi(dx),
\]
$b\in\R$,
$a\ge0$,
$\pi$ is a L\'evy measure on $\R$, that is a $\sigma$-finite Borel measure satisfying
$$\int_{\R}\left(x^2\wedge1\right)\pi(dx)<\infty,$$
with $\pi(\set{0})=0$ for any  $x\in\R$.
The triplet $(a,b,\pi)$ is shortly called the {\it characteristic triplet of $Z$}.

Now we review the construction of integral of non-random function w.\,r.\,t.\ the L\'evy process $Z$ proposed in~\cite{UW} and further developed in \cite{RR}.
Let the interval $[0,T]$ be fixed.
Consider, for any $n$, the partition of $[0,T]$ of the form
$[0,T]=\bigcup_{i=1}^nA_i$,
where $A_i\in\B([0,T])$ and are pair-wise disjoint and $\max_i \lambda(A_i) \rightarrow 0$, $n \to \infty$. Here $\lambda$ denotes the Lebesgue measure on $\mathcal{B}(\mathbb{R})$. Also throughout the paper $\mathcal{B}(\mathcal{S})$ stands for the Borel $\sigma$-field on the measurable space $\mathcal{S}$ equipped with the topology generated by the open sets.
Hereafter we construct a measure on $\mathcal{B}([0,T])$ taking values in the space
$L_0(\Omega,\F,\P)$  as follows.
We take a L\'evy process $Z$ and consider for any $A\in\B([0,T])$ a random variable $Z(A)$ with the characteristic function of the form
\[
\E\exp\set{i\mu Z(A)}
=\exp\set{\lambda(A)\Psi(\mu)}.
\]
Evidently, $Z$ is a measure on $\B([0,T])$ with the values in $L_0(\Omega,\F,\P)$ and
$Z([0,t])=Z_t$ is the value of the L\'evy process $Z$ at point $t$.
Introduce the following definition.

\begin{definition}[\cite{RR}]\label{def:1}
\hspace{5mm}
\begin{itemize}
\item[$(i)$] Let
$f(x)=\sum_{j=1}^nf_j\ind_{A_j}$
be a real-valued simple function on $[0,T]$, where $A_j\in\B([0,T])$ are pair-wise disjoint and
$\bigcup_{j=1}^nA_j=[0,T]$.
Then, for any $A\in\B([0,T])$, we set
\[
\int_Af\,dZ=\sum_{j=1}^nf_jZ(A\cap A_j).
\]
\item[$(ii)$]
A measurable function
$f\colon ([0,T],\B([0,T]))\to(\R,\B(\R))$
is said to be $Z$-integrable if there exists a sequence
$\set{f_n,n\ge1}$ of simple functions as in $(i)$ such that
\begin{enumerate}[1)]
\item $f_n\to f$ $\lambda$-a.\,e.
\item for any $A\in\B([0,T])$ the sequence
$\int_Af_n\,dZ$ converges in probability ($\plim$) as $n\to\infty$.
\end{enumerate}
If $f$ is $Z$-integrable, we put
\[
\int_Af\,dZ=\plim_{n\to\infty}\int_Af_n\,dZ.
\]
\end{itemize}
\end{definition}
The following statement summarises the basic facts about the newly introduced integral. They are established  in \cite{RR} and \cite{UW}. From now on we put $0\cdot\infty=0.$

\begin{proposition}\label{l:levy}\hspace{5mm}
\begin{itemize}
\item[$(i)$]
The integral
$\int_Af\,dZ$ is well  defined, i.e., for any $Z$-integrable function \\ $f\colon ([0,T],\B([0,T])) \rightarrow (\R,\B(\R))$, the integral does not depend on the choice of approximating sequence $\set{f_n,n\ge1}$.
\item[$(ii)$]
Define
\[
r(u):=au^2+\int_{\R}\left(\abs{xu}^2\wedge1\right)\,\pi(dx)
+\abs{bu+\int_{\R}\bigl(\tau(xu)-\tau(x)u\bigr)\,\pi(dx)}.
\]
Then a measurable function
$f\colon ([0,T],\B([0,T]))\to(\R,\B(\R))$ is
$Z$-integrable if and only if
$\int_{[0,T]}r(f(s))\,ds<\infty.$
\item[$(iii)$] If $f$ is $Z$-integrable, then the characteristic function of the integral can be rewritten as the characteristic function of a L\'{e}vy process:
\begin{multline}\label{eq:ch-f}
\E\exp\set{i\lambda\int_{[0,T]}f\,dZ}
=\exp\left\{ \int_{[0,T]}\Psi(\lambda f(s))\,ds\right\}\\
=\exp\set{ib_f\lambda-\frac{a_f\lambda^2}{2}
+\int_{\R}\left(e^{i\lambda x}-1-i\lambda\tau(x)\right)F_f(dx)},
\end{multline}
where
\begin{align*}
b_f&=\int_{[0,T]}\left(bf(s)+\int_{\R}\bigl(\tau(xf(s))-\tau(x)f(s)\bigr)\,\pi(dx)\right)ds,\\
a_f&=\int_{[0,T]}a f^2(s)\,ds,\\
F_f(B)&=\int_{[0,T]}\int_{\R}\ind_{f(s)x\in B\setminus\set{0}}\,\pi(dx)\,ds,
\quad B\in\B(\R).
\end{align*}
\end{itemize}
\end{proposition}

The next lemma contains the modifications of well known properties of the introduced integral in the form that is suitable for our further considerations.

\begin{lemma}\label{l:levy2}
\hspace{5mm}
\begin{itemize}
\item[$(i)$]
Any function $f\in L_2([0,T])$ is $Z$-integrable.
In this case the characteristic function of the integral has the following form
\begin{multline}\label{eq:ch-f-2}
\E\exp\set{i\lambda\int_{[0,T]}f\,dZ}
=\exp\left\{ib\lambda\int_{[0,T]}f(s)\,ds
-\frac12a\lambda^2\int_{[0,T]}f^2(s)\,ds\right.\\
+\left.\int_{[0,T]}\left(\int_\R\left(e^{i\lambda f(s)x}-1-i\lambda f(s)\tau(x)\right)\pi(dx)\right)ds\right\}.
\end{multline}
\item[$(ii)$]
Let $p\in[1,2)$.
Suppose that $Z$  satisfies the additional assumptions:   $a=0$ and
$\int_{\abs{x}\le1}\abs{x}^p\,\pi(dx)<\infty$.
 Then
any function $f\in L_p([0,T])$ is $Z$-integrable, and
\begin{multline}\label{eq:ch-f-p}
\E\exp\set{i\lambda\int_{[0,T]}f\,dZ}
=\exp\left\{ib\lambda\int_{[0,T]}f(s)\,ds\right.\\
+\left.\int_{[0,T]}\left(\int_\R\left(e^{i\lambda f(s)x}-1-i\lambda f(s)\tau(x)\right)\pi(dx)\right)ds\right\}.
\end{multline}
\end{itemize}
\end{lemma}

\begin{proof}
According to paragraph $(ii)$ from Proposition~\ref{l:levy}, in order to establish $Z$-integrability, we need to prove that
\begin{gather*}
\int_{[0,T]}\int_{\R}a f^2(s)\,\pi(dx)\,ds<\infty,\\
\int_{[0,T]}\int_{\R}\left(\abs{xf(s)}^2\wedge1\right)\,\pi(dx)\,ds<\infty,
\intertext{and}
\int_{[0,T]}\int_{\R}\abs{\tau(xf(s))-\tau(x)f(s)}\,\pi(dx)\,ds<\infty.
\end{gather*}
The first integral is finite, since $f\in L_2([0,T])$ in the case $(i)$ and $a=0$ in the case $(ii)$.

Recall that $\int_\R\left(x^2\wedge1\right)\pi(dx)<\infty$ by the definition of the L\'evy measure.
Then
\[
\int_{\abs{x}\le1}\abs{x}^p\pi(dx)<\infty
\quad\text{and}\quad
\int_{\abs{x}>1}\pi(dx)<\infty
\]
in both cases $(i)$ and $(ii)$.
Since the rest of proof can be carried out similarly for both statements, from now we assume that $p\in[1,2]$.
Consider the second integral
\begin{align*}
\int_{[0,T]}\int_{\R}&\left(\abs{xf(s)}^2\wedge1\right)\,\pi(dx)\,ds
\le\int_{[0,T]}\int_{\R}\left(\abs{xf(s)}^p\wedge1\right)\,\pi(dx)\,ds\\
&\le\int_{[0,T]}\int_{\abs{x}\le1}\abs{xf(s)}^p\,\pi(dx)\,ds
+\int_{[0,T]}\int_{\abs{x}>1}\pi(dx)\,ds\\
&=\int_{[0,T]}\abs{f(s)}^p\,ds\int_{\abs{x}\le1}\abs{x}^p\,\pi(dx)
+T\int_{\abs{x}>1}\pi(dx)
<\infty.
\end{align*}
The third integral can be  rewritten as follows:
\begin{align*}
\int_{[0,T]}\int_{\R}&\abs{\tau(xf(s))-\tau(x)f(s)}\,\pi(dx)\,ds\\
&=\int_{[0,T]}\int_{\abs{x}\le1}\abs{\tau(xf(s))-\tau(x)f(s)}\ind_{\abs{xf(s)}\le1}\,\pi(dx)\,ds\\
&+\int_{[0,T]}\int_{\abs{x}\le1}\abs{\tau(xf(s))-\tau(x)f(s)}\ind_{\abs{xf(s)}>1}\,\pi(dx)\,ds\\
&+\int_{[0,T]}\int_{\abs{x}>1}\abs{\tau(xf(s))-\tau(x)f(s)}\,\pi(dx)\,ds
=:I_1+I_2+I_3.
\end{align*}
Note that $I_1=0$, because $\tau(xf(s))-\tau(x)f(s)=0$ for $\abs{x}\le1$ and $\abs{xf(s)}\le1$.
Let $\abs{x}\le1$ and $\abs{xf(s)}>1$. Then
\begin{align*}
\abs{\tau(xf(s))-\tau(x)f(s)}
&=\abs{\frac{xf(s)}{\abs{xf(s)}}-xf(s)}\\
&=\abs{xf(s)}\abs{\frac{1}{\abs{xf(s)}}-1}
\le2\abs{xf(s)}
\le2\abs{xf(s)}^p.
\end{align*}
Hence,
\[
I_2\le2\int_{[0,T]}\abs{f(s)}^p\,ds\int_{\abs{x}\le1}\abs{x}^p\,\pi(dx)<\infty.
\]
Finally, using the inequality $\abs{\tau(z)}\le1$, we can write
\[
\abs{\tau(xf(s))-\tau(x)f(s)}
\le\abs{\tau(xf(s))}+\abs{\tau(x)f(s)}
\le1+\abs{f(s)}.
\]
Then
\[
I_3\le\int_{[0,T]}(1+\abs{f(s)})\,ds\int_{\abs{x}>1}\pi(dx)<\infty.
\]
This concludes the proof of $Z$-integrability.
The formulas \eqref{eq:ch-f-2}--\eqref{eq:ch-f-p} for the characteristic functions follow directly from the statement $(iii)$ of Proposition~\ref{l:levy}.
\end{proof}

\begin{remark} With no  doubt, any function  $f\in L_p([0,T]), p>2$ is $Z$-integrable.
\end{remark}

The rest of this section is devoted to the upper bounds for the moments of the integral, which are fundamental tools for the analysis in the sequel.

First, assume that $f\in L_2([0,T])$ and $\int_{\R}x^2\pi(dx)<\infty$.
Then by differentiation of the characteristic function~\eqref{eq:ch-f-2}, one can deduce that the integral admits second moment. In fact we have
\begin{align*}
\E\abs{\int_{[0,T]}f\,dZ}^2
=&\left(\int_{[0,T]}f(s)\,ds\right)^2
\left(b+\int_{\R}(x-\tau(x))\,\pi(dx)\right)^2\\
&+\int_{[0,T]}f^2(s)\,ds
\left(a+\int_{\R}x^2\pi(dx)\right)<\infty.
\end{align*}
Here we use that
$\int_{\R}\abs{x-\tau(x)}\,\pi(dx)=\int_{\abs{x}>1}\abs{x-\tau(x)}\,\pi(dx)\\
\le2\int_{\abs{x}>1}\abs{x}\,\pi(dx)
\le2\int_{\abs{x}>1}x^2\,\pi(dx)<\infty$.

In the case when  $b=0$ and the measure $\pi$ is symmetric, the formula for the second moment is  simplified.
Indeed, in this case $\int_{\R}(x-\tau(x))\,\pi(dx)=0$
and
\[
\E\abs{\int_{[0,T]}f\,dZ}^2
=\int_{[0,T]}f^2(s)\,ds
\left(a+\int_{\R}x^2\pi(dx)\right).
\]

Now let us consider the general case  $p\ge1$.
The following theorem gives an a priori estimate for the $p$th moment of the integral. By Lemma~\ref{l:levy2}, in order to integrate functions from $L_p([0,T])$ with $p\in[1,2)$ we need to assume that $a=0$ for the process $Z$.

\begin{theorem}\label{l:apriori-p}
\hspace{5mm}
\begin{itemize}
\item[$(i)$]
Let $p\in[1,2)$.
Assume that $f\in L_p([0,T])$ and that the characteristic triplet of $Z$ satisfies
$a=b=0$, $\pi$ is symmetric,
$\int_\R\abs{x}^p\,\pi(dx)<\infty$.
Then
\begin{equation}\label{eq:apr1}
\E\abs{\int_{[0,T]}f\,dZ}^p
\le C\norm{f}_{L_p([0,T])}^p\int_\R \abs{x}^p\,\pi(dx).
\end{equation}
\item[$(ii)$]
Let $p\ge2$.
Assume that $f\in L_p([0,T])$ and that
$b=0$, $\pi$ is symmetric and\linebreak
$\int_\R\abs{x}^p\,\pi(dx)<\infty$.
Then
\begin{equation}\label{eq:apr2}
\E\abs{\int_{[0,T]}f\,dZ}^p
\le C\left(a^{p/2}\norm{f}_{L_2([0,T])}^p+\norm{f}_{L_p([0,T])}^p\int_\R \abs{x}^p\,\pi(dx)\right).
\end{equation}
\end{itemize}
\end{theorem}

\begin{proof}
$(i)$
Let
\[
\Phi_p(u)=\int_\R\abs{ux}^p\,\pi(dx).
\]
Evidently, $\Phi_p\colon\R\to\R_+$ is a Young function, i.\,e. it is a convex function such that $\Phi_p(u)=\Phi_p(-u)$, $\Phi_p(0)=0$ and
$\lim_{u\to\infty}\Phi_p(u)=\infty$.
Therefore, we can consider the Orlicz space
\[
L_{\Phi_p}([0,T])
=\set{f\in L_0([0,T]) : \int_{[0,T]}\Phi_p(\abs{f(s)})\,ds<\infty}=L_p([0,T])
\]
with the Luxemburg norm
\begin{equation}\label{eq:norm}
\begin{split}
\norm{f}_{\Phi_p}&=\inf\set{c>0 : \int_{[0,T]}\Phi_p\left(c^{-1}\abs{f(s)}\right)\,ds\le1}\\
&= {\norm{f}_{L_p([0,T])}}{\left(\int_\R\abs{x}^p\,\pi(dx)\right)^{\frac{1}{p}}}.
\end{split}
\end{equation}
Then $L_{\Phi_p}([0,T])$ obviously is a Banach space.

First, let us prove that $\int_{[0,T]}f\,dZ\in L_p(\Omega;\P)$ for any $f\in L_{\Phi_p}([0,T])$.
Assume that $f\in L_{\Phi_p}([0,T])$, that is
$\int_{[0,T]}\Phi_p(\abs{f(s)})\,ds<\infty$.
Recall that by $F_f(\cdot)$ the L\'evy measure in the canonical representation of the characteristic function of
$\int_{[0,T]}f\,dZ$ (see Proposition~1 (iii)).
Then, by Proposition~1,
\begin{align*}
\int_{\abs{u}>1}\abs{u}^p\,F_f(du)
&=\int_{[0,T]}\int_{\set{\abs{f(s)x}>1}}\abs{f(s)x}^p\,\pi(dx)\,ds\\
&\le\int_{[0,T]}\Phi_p(\abs{f(s)})\,ds<\infty.
\end{align*}
Taking into account paragraph $(iii)$ of Proposition \ref{l:levy} and the well-known property of L\'evy processes (see \cite[Theorem~25.3]{sato}), we can conclude, as it was done in the proof of inequality (3.6), Theorem 3.3 from  \cite{RR}, that the finiteness of this integral implies the finiteness of
$\E\abs{\int_{[0,T]}f\,dZ}^p$.

Further, let us prove that the linear mapping
\[
L_{\Phi_p}([0,T])\ni f\longmapsto
\int_{[0,T]}f\,dZ\in L_p(\Omega;\P)
\]
is continuous.
Let $f_n\to0$ in $L_{\Phi_p}([0,T])$.
This implies that
\begin{equation}\label{eq:conv-1}
\int_{[0,T]}\Phi_p(\abs{f_n(s)})\,ds\to0,
\quad\text{as }n\to\infty,
\end{equation}
see \cite[Proposition~3.2.4]{RaoRen91}.
Let $b_n$, $a_n$ and $F_n$ be, respectively, the centring constant, the variance, and the L\'evy measure in the canonical representation of the characteristic function of
$\int_{[0,T]}f_n\,dZ$ (see Proposition 1 (iii)).
Under the assumptions taken, being $\pi$ symmetric, we have $a_n=b_n=0$, and
\begin{align*}
\int_{\abs{u}>1}\abs{u}^p\,F_n(du)
&=\int_{[0,T]}\int_{\abs{f_n(s)x}>1}\abs{f_n(s)x}^p\,\pi(dx)\,ds\\
&\le\int_{[0,T]}\int_{\R}\abs{f_n(s)x}^p\,\pi(dx)\,ds\rightarrow 0,
\end{align*}
as $n\to\infty$, by~\eqref{eq:conv-1}.
Then the convergence
\[
\E\abs{\int_{[0,T]}f_n\,dZ}^p\rightarrow0,
\quad\text{as }n\to\infty,
\]
follows from~\cite[Lemma~3.2]{RR}.
Thus, the continuity is proved.

Since any continuous linear operator is bounded \cite[\S~29, Theorem 1]{KolmFomin1}, we have
\[
\left(\E\abs{\int_{[0,T]}f\,dZ}^p\right)^{\frac1p}\le C\norm{f}_{\Phi_p},
\]
where the constant $C$ does not depend on $f$. Taking \eqref{eq:norm} into account,
we conclude the proof.

$(ii)$ The statement can be proved similarly, using the function
\[
\Phi_p^{(a)}(u)=au^2+\int_\R\abs{ux}^p\,\pi(dx)
\]
instead of $\Phi_p$.
Arguing as above, we get
\[
\E\abs{\int_{[0,T]}f\,dZ}^p\le C\norm{f}_{\Phi_p^{(a)}}^p.
\]
Then it is not hard to see that
\begin{align*}
\norm{f}_{\Phi_p^{(a)}}^p
&=\left(\inf\set{c>0:\frac{a\norm{f}_{L_2([0,T])}^2}{c^2}
+\frac{\norm{f}_{L_p([0,T])}^p\int_\R \abs{x}^p\,\pi(dx)}{c^p}\le1}\right)^p\\
&\le C_1\left(a^{1/2}\norm{f}_{L_2([0,T])}+\norm{f}_{L_p([0,T])}\left(\int_\R \abs{x}^p\,\pi(dx)\right)^{1/p}\right)^p\\
&\le C_2\left(a^{p/2}\norm{f}_{L_2([0,T])}^p+\norm{f}_{L_p([0,T])}^p\int_\R \abs{x}^p\,\pi(dx)\right).
\end{align*}
\end{proof}

\begin{remark}
The case $b\ne0$ can be considered similarly.
If the other assumptions of the above theorem hold, then
\[
\E\abs{\int_{[0,T]}f\,dZ}^p
\le C\left(\abs{b}^p\norm{f}_{L_1([0,T])}^p+\norm{f}_{L_p([0,T])}^p\int_\R \abs{x}^p\,\pi(dx)\right).
\]
for $p\in[1,2)$, and
\begin{multline}\label{ineq-powers}
\E\abs{\int_{[0,T]}f\,dZ}^p
\le C\biggl(\abs{b}^p\norm{f}_{L_1([0,T])}^p+a^{p/2}\norm{f}_{L_2([0,T])}^p\\
+\norm{f}_{L_p([0,T])}^p\int_\R \abs{x}^p\,\pi(dx)\biggr).
\end{multline}
for $p\ge2$.
In this case the functions $\Phi_p$ and $\Phi_p^{(a)}$ in the proof are replaced with
$\Phi_p^{(0,b)}(u)=\abs{bu}+\Phi_p(u)$ and
$\Phi_p^{(a,b)}(u)=\abs{bu}+\Phi_p^{(a)}(u)$, respectively.
\end{remark}

\begin{remark} The upper bound \eqref{ineq-powers} can be simplified to
\begin{equation}\label{ineq-p}
\E\abs{\int_{[0,T]}f\,dZ}^p
\le C \norm{f}_{L_p([0,T])}^p.
\end{equation}
Inequality \eqref{ineq-p} for predictable stochastic integrands is contained in Theorem 66 \cite{Prott}. For such integrands the inequality \eqref{ineq-p} is a partial case of Bichteler--Jacod inequality, see, e.g., \cite{MarRock}.  However, we prefer to give here the detailed structure  of the right-hand side.
\end{remark}

\begin{remark}\label{rem-mart}
{\bf Volterra processes driven by square integrable martingales}.
Let $p\geq 2$ and $\int_\R \abs{x}^p\,\pi(dx)<\infty$. Then, taking into account inequality $\int_{\R}\left(x^2\wedge1\right)\pi(dx)<\infty,$  we get that $\int_\R x^2\,\pi(dx)<\infty$, so that $Z$ is a square-integrable process.
Assuming additionally that $b=0$, and the measure $\pi$ is symmetric, we can see that $Z$ is a square-integrable martingale with quadratic characteristics $\langle Z\rangle_t=(a+\int_\R x^2\,\pi(dx))t.$  However,   we can consider the general case.
Indeed, let $M$ be a square integrable c\`{a}dl\`{a}g martingale with quadratic variation $[M]$ and zero mean. Then, using $M$ as integrator, the construction of the integral $\int_0^T fdM$ coincides with the one of $\int_0^T fdZ$, constructed above.
Furthermore, according to Burkholder-Davis-Gundy inequalities, for any $p\geq 1 $ there exists a constant $C=C_p$ such that
$$
\E\left|\int_0^T fdM\right|^p\leq C_p\E\left(\int_0^T f^2d[M]\right)^\frac{p}{2}.
$$
In the simplest case, when $p=2$, we obtain that $$\E\left|\int_0^T fdM\right|^2\leq C_p\E\left(\int_0^T f^2d[M]\right)=C_p\E\left(\int_0^T f^2d\langle M\rangle\right).
$$
If $\langle M\rangle_t=\int_0^t m_sds $ with $\E m_s\leq C$, we obtain the same bound as \eqref{ineq-p}, but for a wider class of processes. So, in this case we can use the martingale approach instead of the L\'{e}vy-processes approach.

 Comparing our a priori estimates with estimates for other classes of   integrators, we can consider  the process having the form of the sum
 $$
 Y_t=\int g_1dM+\int g_2d\overline{\mu}=\int_0^t g_1dM+\int_{(0,t] \times\R}g_2(s,z)\overline{\mu}(ds,dz),
 $$
 where $M$ is a square-integrable continuous martingale with quadratic  characteristics $\langle M\rangle$, $\overline{\mu}=\mu-\nu$, $\mu$ is a square-integrable random measure with dual predictable projection $\nu$, integrands $g_i, i=i,2$ are predictable and such that all integrals are well-defined and square-integrable. Then, according to Burkholder-Davis-Gundy inequalities and the generalisation of Bichteler--Jacod inequalities from \cite{MarRock},  the following estimate holds: for any $T>0$, $\alpha\in[1,2]$ and any $p\geq 1$ there exists a constant $C=C_{\alpha,p,T}$ such that
\begin{equation*}\E(\sup_{t\in[0,T]}|Y_t|^p)\leq C_{\alpha,p,T}\left(\E\left(\int_{[0,T]}g_1^2d\langle M\rangle\right)^{\frac{p}{2}}+\E\left(\int_{[0,T]\times \R}
|g_2|^\alpha d\nu\right)^{\frac{p}{\alpha}}\right),\end{equation*}
 for $p\in[1, \alpha]$,and
\begin{equation*}\begin{gathered}\E(\sup_{t\in[0,T]}|Y_t|^p)\leq C_{\alpha,p,T}\Bigg(\E\left(\int_{[0,T]}g_1^2d\langle M\rangle\right)^{\frac{p}{2}}+\E\left(\int_{[0,T]\times \R}|g_2|^\alpha d\nu\right)^{\frac{p}{\alpha}}\\+\E\left(\int_{[0,T]\times \R}|g_2|^p d\nu\right)\Bigg) \end{gathered}\end{equation*}
for $p\in(\alpha, \infty).$
However, we shall not consider such processes in the framework of the present paper.
\end{remark}

   \vspace*{-12pt} 
\subsection{Integration of Volterra-type kernels with respect to a L\'{e}vy  process}\label{sec:int-2}
Now, let us have a two-parameter measurable non-random kernel of the form $g=g(t,s)\colon\R_+^2\to\R$,
and our goal is to construct the integral
$J(t,g)=\int_0^tg(t,s)\,dZ_s$, for any $t\in[0,T]$.
This construction is the same as for constructed   in Subsection \ref{sec:int-1} integral of non-random functions $f$, therefore we can use Lemma \ref{l:levy2} and Theorem \ref{l:apriori-p} and immediately proceed with the following conclusion.

\begin{theorem}
\label{th:2}
Let one of the following conditions hold:
\begin{itemize}
\item[$(A)$] for some $p\in[1,2)$,  $g=g(t,\cdot)\in L_p([0,t])$ for any $t\in[0,T]$,
$a=b=0$, and the measure $\pi$ is symmetric with
$\int_\R\abs{x}^p\,\pi(dx)<\infty$;
\item[$(B)$] for some $p\ge2$, $g=g(t,\cdot)\in L_p([0,t])$ for any $t\in[0,T]$, $b=0$,  and the measure $\pi$ is symmetric with  $\int_{\R}\abs{x}^p\pi(dx)<\infty$.
\end{itemize}
Then, for any $t\in[0,T]$, $g(t,\cdot)$ is $Z$-integrable,   in the case when condition $(A)$ holds, we have the a priori estimate
\begin{equation}\label{estimator3}
\E\abs{\int_0^tg(t,s)\,dZ_s}^p
\le C\norm{g(t,\cdot)}_{L_p([0,t])}^p \int_{\R}|x|^p\pi(dx),
\end{equation}
and   in the case when condition $(B)$ holds, we have the a priori estimate
 \begin{multline}\label{estimator2}
\E\abs{\int_0^tg(t,s)\,dZ_s}^p\\
\le C\left(a^{p/2}\norm{g(t,\cdot)}_{L_2([0,t])}^p+\norm{g(t,\cdot)}_{L_p([0,t])}^p\int_{\R}\abs{x}^p\pi(dx)\right).
\end{multline}
\end{theorem}

\begin{remark}\label{rem kern}
\hspace{5mm}
\begin{itemize}
\item[$(i)$] It is sufficient for our purposes to consider  the restriction of $g$  to the set $\{0\le s<t\le T\}$. We  can simply assume that $g:\{0\le s<t\le T\}\rightarrow \R.$
 \item[$(ii)$] The extension to square-integrable martingale considered in Remark \ref{rem-mart} is valid in the case of the kernel $g$ with evident corrections. \end{itemize}
\end{remark}

\section{An example of L\'{e}vy process as integrator: the subordinated  Wiener process}\label{sec:SWP}

\setcounter{section}{3}
\setcounter{equation}{0}\setcounter{theorem}{0}

Time change and here, in particular, subordination is a feasible way to build L\'evy processes from known ones. This constitute one of the simplest ways to simulation and thus it gains particular interest.
In this section we shall concentrate on this case. Se e.g. \cite{CT}.

\subsection{Description of subordinate Wiener process}
Let $W=\left\{W_t, t\ge 0\right\}$ be a one-dimensional Wiener process. Subordination of the Wiener process consists in time-changing the paths of $W$ by  an independent subordinator $L=\left\{L_t, t\ge0\right\}$, which is a non-negative, non-decreasing L\'{e}vy process starting from $0$.
The Laplace exponent $\Phi=\Phi(\lambda)$ of $L$, defined by the relation
$$\E\exp\left\{-\lambda L_t\right\}=\exp\left\{-t\Phi(\lambda)\right\}, \,\,\, \lambda>0,$$
has the form
$$\Phi(\lambda)=a\lambda+\int\limits_0^{\infty}(1-e^{-\lambda x})\nu(dx),$$
where $a>0$ is the drift of the subordinator and $\nu$ is its L\'{e}vy measure with $\int\limits_0^\infty(1\wedge x)\nu(dx)<\infty$.

\vspace{2mm}
Consider the  function $(2\pi s)^{-\frac{1}{2}}\exp\left\{-\frac{x^2}{2s}\right\}$ which    is bounded in $s$ on $\R^+=(0,+\infty)$ for any fixed $x\in\R $.   Introduce the following density function
$$
\varrho(x)=\int\limits_0^\infty(2\pi s)^{-\frac{1}{2}}\exp\left\{-\frac{x^2}{2s}\right\}\nu(ds), \quad x\in\R,
$$
and let $\pi$ be a measure on $\mathcal{B}(\mathbb{R} )$ with density $\varrho$.
For later use we introduce the following condition:
\begin{itemize}
\item [(C)] $\int\limits_0^{\infty}x^{\frac{1}{2}}\nu(dx)<\infty.$
\end{itemize}

\begin{lemma} \label{theor1:1} The following statements are true.
\begin{itemize}
\item[$(i)$] The subordinate Wiener process
$$W^L:=W(L)=\left\{W^L_t:=W(L_t),\,\,\, t\ge 0\right\}$$
characteristic function
\begin{gather}
\notag
\Upsilon(\mu):=\E\exp\left\{i\mu W^L_t\right\}=\E\exp\left\{-\frac{\mu^2}{2}L_t\right\}=\exp\left\{-t\Phi\left(\frac{\mu^2}{2}\right)\right\}\\
\label{ch_f}
=\exp\left\{-t\left(\frac{a\mu^2}{2}+\int\limits_0^\infty\left(1-e^{-\frac{x\mu^2}{2}}\right)\nu(dx)\right)\right\}.
\end{gather}
\item[$(ii)$] The subordinated Wiener process $W^L$ is a L\'evy process with zero drift coefficient, its   diffusion coefficient equals $a$,  and  L\'{e}vy measure equal $\pi$. Its characteristic function can be represented as
\begin{gather}\label{equ1}
 \E\exp\left\{i\mu W^L_t\right\}=
\exp\left\{t\left(-\frac{a\mu^2}{2}+ \int_\R\left( e^{i\mu x}-1-i\mu x1_{|x|<1}\right)\pi(dx)\right)\right\}.
\end{gather}
\item[$(iii)$]  Let condition $(C)$ hold. Then $\int\limits_{\R}  |x| \pi(dx)<\infty $, and therefore, $$\E|W^L_t|=t\E|W^L_1|<\infty$$ for any $t\geq 0$.
\end{itemize}
\end{lemma}

\begin{proof}Statements $(i)$ and $(ii)$  immediately follow  from  Theorem 30.1 \cite{sato}. To prove $(iii)$, note that
\begin{equation}\label{eq:pi_nu}
\begin{split}
\int_\R\abs{x}\pi(dx)
&=\int_\R\abs{x}\varrho(x)\,dx
=\int_\R\int_0^\infty\abs{x}(2\pi s)^{-\frac{1}{2}}\exp\left\{-\frac{x^2}{2s}\right\}\nu(ds)\,dx\\
&=2(2\pi)^{-\frac12}\int_0^\infty s^{\frac12}\nu(ds) \int_0^\infty ze^{-\frac{z^2}{2}}\,dz\\
&=2(2\pi)^{-\frac12}\int_0^\infty s^{\frac12}\nu(ds)
<\infty.
\end{split}
\end{equation}
\end{proof}

\begin{remark} Note that    the density $\varrho$ is a symmetric  function therefore  the following equality holds: $\int_\R x1_{|x|<1} \pi(dx)=0$, and we can rewrite \eqref{equ1} as
\begin{gather} \label{equ2}
 \E\exp\left\{i\mu W^L_t\right\}=\exp\left\{t\Psi(\mu)\right\}:=
\exp\left\{t\left(-\frac{a\mu^2}{2}+ \int_\R\left( e^{i\mu x}-1\right)\pi(dx)\right)\right\}.
\end{gather}
\end{remark}

Here below we can consider three particular cases.

\subsubsection{Subordinate  Wiener process as a  square integrable martingale}\label{sssec1.1.1}

Introduce the natural  filtration  $\F^{W^L}=\set{\F_s^{W^L},s\ge0}$,
where
$\F_s^{W^L}=\sigma\set{W^L_u,0\le u\le s}$. Introduce the condition
\begin{itemize}
\item [(D)] $\int\limits_0^{\infty}x \nu(dx)<\infty.$
\end{itemize}
Condition $(D)$ is equivalent to the existence of the expectation of $L_t$ for any $t\geq 0$, because $$\E L_t=t\E L_1=t\Phi'(\lambda)\big|_{\lambda=0}=t(a+\int_0^{\infty}x \nu(dx)).$$
\begin{lemma}\label{lem1}
Under condition $(D)$ we have that $\int_{\R}x^2\pi(dx)<\infty$ and the process $W^L$ is a square-integrable martingale w.r.t. the natural filtration with the quadratic characteristic $\langle W^L\rangle_t=ct$, where $c=\E L_1.$
 \end{lemma}

\begin{proof}Inequality $\int_{\R}x^2\pi(dx)<\infty$ is established similarly to statement $(iii)$ of Lemma \ref{theor1:1}.
The second statement is also easy to prove.
Let $\E L_1<\infty$.
Then it follows immediately from \eqref{ch_f} that
 \[
 \E W^L_1=\frac{1}{i}\Upsilon'(\mu)\big|_{\mu=0}=-\frac{1}{i}\mu\E L_1\big|_{\mu=0}=0,
 \]
and \[
\E (W^L_1)^2
 =\E L_1
\]
Namely, $W^L$ is a square integrable L\'evy process $W^L$ with zero-mean,
hence it is a martingale (see Proposition 3.17 in \cite{CT}).
By this we complete the proof.
\end{proof}

As an illustration, consider the Gamma subordinator $L$, which is a L\'evy process $L$ with zero drift and L\'evy measure of the form
\[
\nu(dx)=cx^{-1}e^{-\lambda x}\ind_{x\in\R^+}\,dx.
\]
The corresponding subordinate Wiener process $W^L$ has no diffusion part, conditions~(C) and ~ (D) hold, and it has L\'evy measure $\pi$ with the density of the form
\[
\varrho(x)=c(2\pi)^{-\frac12}\int_0^\infty s^{-\frac32}e^{-\lambda s-\frac{x^2}{2s}}\,ds.
\]
Evidently, in this case  $W^L$ is a square-integrable martingale.

\subsubsection{Subordinate  Wiener process with compound Poisson subordinator}

Let the L\'evy process $L$ be a compound Poisson process, that is equivalent to the fulfilment of the conditions $a=0$ and $\nu(\R^+)<\infty$.
In this case the subordinate Wiener process $W^L$ is a L\'evy process without diffusion component, with characteristic function
\[
\E\exp\set{i\mu W^L_t}=\exp\set{t\int_\R\left(e^{i\mu x}-1\right)\pi(dx)},
\]
and
\begin{align*}
\pi(\R)&=\int_\R\varrho(x)\,dx
=\int_\R\int_{\R^+}(2\pi s)^{-\frac12}e^{-\frac{x^2}{2s}}\,\nu(ds)\,dx\\
&=\int_{\R^+}\nu(ds)=\nu\left(\R^+\right)<\infty.
\end{align*}
Therefore, $W^L$ is a compound Poisson process.

\subsubsection{Subordinate  Wiener process with a stable subordinator}

Consider a measure $\nu_\alpha(dx)$ on $\R^+$ of the form
\[
\nu_\alpha(dx)=\frac{c}{x^{1+\alpha}}\ind_{\set{x\in\R^+}}\,dx.
\]
Then $\nu_\alpha$ is the L\'evy measure of some L\'evy process if and only if $\alpha\in(0,2)$, yet $\nu_\alpha$ is the L\'evy measure of some subordinator if and only if $\alpha\in(0,1)$.
So, a stable subordinator $L$ with index $\alpha\in(0,1)$ is a subordinator with zero drift and L\'evy measure $\nu_\alpha$.
Its moment generating function is given by
\[
\E e^{-\lambda L_t}=e^{-c_1t\lambda^\alpha},
\quad \lambda,t>0.
\]
In this case the subordinate Wiener process $W^L$ is $2\alpha$-stable and has a characteristic function of the form
\[
\E\exp\set{i\mu W^L_t}=\E\exp\set{-\frac{\mu^2}{2}L_t}
=\exp\set{-\frac{c_1}{2^\alpha}\mu^{2\alpha}t}.
\]
The moments $\E{L^\beta_t}$, $\beta>0$, for a stable subordinator with index $\alpha$ exist only for $\beta<\alpha$ and
\[
\E L^\beta_t =\frac{(c_1t)^{\frac{\beta}{\alpha}}\Gamma\left(1-\frac{\beta}{\alpha}\right)}{\Gamma(1-\beta)}.
\]
Therefore, for $\alpha\in(\frac12,1)$,
$\E{L_t}^{1/2}<\infty$ and
consequently  $\E\abs{W^L_t}<\infty$.
For any $\alpha\in(0,1)$
$\E (W^L)_t^2=\infty$.

\subsection{Integration of a non-random kernel with respect to a subordinate Wiener process}

Now we apply to a subordinate Wiener process $W^L$ the construction of integral w.\,r.\,t.\ a L\'evy process $Z$ from Section~\ref{sec:int}.
By Lemma~\ref{theor1:1}, $W^L$ is a L\'evy process with zero drift coefficient, the diffusion coefficient $a$, and the symmetric L\'evy measure $\pi$.
Similarly to \eqref{eq:pi_nu}, one can show that
\[
\int_\R\abs{x}^p\pi(dx)
\le C\int_0^\infty s^{p/2}\nu(ds).
\]
Then from Theorem~\ref{l:apriori-p} we immediately get the following result.

\begin{lemma}\hspace{5mm}
\begin{itemize}
\item[$(i)$]
Let $p\in[1,2)$, $f\in L_p([0,T])$, $a=0$,
and $\int_0^\infty s^{p/2}\nu(ds)<\infty$.
Then $f$ is $W^L$-integrable and
\begin{equation}\label{charact1:1a}
\E\abs{\int_{[0,T]}f\,dW^L}^p
\le C\norm{f}_{L_p([0,T])}^p\int_0^\infty s^{p/2}\nu(ds).
\end{equation}
\item[$(ii)$]
Let $p\ge2$, $f\in L_p([0,T])$,
and $\int_0^\infty s^{p/2}\nu(ds)<\infty$.
Then $f$ is $W^L$-integrable and
\begin{equation}\label{charact1:1b}
\E\abs{\int_{[0,T]}f\,dW^L}^p
\le C\left(a^{p/2}\norm{f}_{L_2([0,T])}^p+\norm{f}_{L_p([0,T])}^p\int_0^\infty s^{p/2}\nu(ds)\right).
\end{equation}
\end{itemize}
\end{lemma}

\begin{remark}\label{rem5}
Let $L_t\in L_1(\P)$ for any $t\in[0,T]$ so that $W^L_t\in L_2(\P)$ for any $t\in[0,T]$, and $W^L$ is a square-integrable martingale.
We can create the sequence of partitions
$\pi_n=\set{0=t_0^n<t_1^n<\ldots<t_{k_n}^n=T}$
with $\diam\pi_n\to0$, $n\to0$ and choose in Definition~\ref{def:1} the sets $A_j^n=[t_j^n,t_{j+1}^n)$.
Then for any $f\in L_2([0,T])$ we see that
$\int_{[0,T]}f\,dW^L$
coincides with the Wiener integral
$\int_0^Tf(s)\,dW^L_s$
of the non-random function w.\,r.\,t.\ a square-integrable martingale.
\end{remark}

Now, let us have a two-parameter measurable non-random kernel of the form $g=g(t,s)\colon\R_+^2\to\R$.
Using Theorem~\ref{th:2}, we obtain the following statement.

\begin{lemma}\hspace{5mm}
\begin{itemize}
\item[$(i)$]
Let for some $p\in[1,2)$  $g=g(t,\cdot)\in L_p([0,t])$ for any $t\in[0,T]$,
$a=0$, and
$\int_0^\infty s^{p/2}\nu(ds)<\infty$.
Then for any $t\in[0,T]$ $g(t,\cdot)$ is $W^L$-integrable, and
\begin{equation}\label{charact2:1a}
\E\abs{\int_{[0,t]}g(t,s)\,dW^L_s}^p
\le C\norm{g(t,\cdot)}_{L_p([0,t])}^p\int_0^\infty s^{p/2}\nu(ds).
\end{equation}
\item[$(ii)$]
Let for some $p\ge2$  $g=g(t,\cdot)\in L_p([0,t])$ for any $t\in[0,T]$, and\linebreak
$\int_0^\infty s^{p/2}\nu(ds)<\infty$.
Then for any $t\in[0,T]$ $g(t,\cdot)$ is $W^L$-integrable, and
\begin{multline}\label{charact2:1b}
\E\abs{\int_{[0,t]}g(t,s)\,dW^L_s}^p\\
\le C\left(a^{p/2}\norm{g(t,\cdot)}_{L_2([0,t])}^p+\norm{g(t,\cdot)}_{L_p([0,t])}^p\int_0^\infty s^{p/2}\nu(ds)\right).
\end{multline}
\end{itemize}
\end{lemma}

\begin{remark}\label{rem6}
Let $g=g(t,\cdot)\in L_2([0,t])$ for any $t\in[0,T]$,
and let the subordinate Wiener process satisfy condition $(D)$.
Then for any $t\in[0,T]$ $g(t,\cdot)$ is $W^L$-integrable, and
\begin{equation}\label{estimator4}
\E\abs{\int_0^tg(t,s)\,dW^L_s}^2
\le C\norm{g(t,\cdot)}_{L_2([0,t])}^2\int_0^\infty x\,\nu(dx).
\end{equation}
For any $t\in[0,T]$, the integral $\int_0^tg(t,s)\,dW^L_s$  coincides with the integral of $g(t,\cdot)$ w.\,r.\,t.\ a square-integrable martingale $W^L$.
Moreover, as it will be clarified by the calculations in the sequel, when $W^L$ is a square-integrable  martingale, it is not really important that $W^L$ is a L\'evy process and $\langle W^L\rangle_t=ct$ with $c=\E L_1$. It is actually important that $W^L$ is a square integrable martingale.
Indeed, wWe can consider any square-integrable martingale
$M=\set{M_t,\F_t,t\in[0,T]}$
with
$\langle M\rangle_t=\int_0^t\sigma^2(s)\,ds$,
where $\sigma$ is a random measurable adapted function with bounded expectation, $\E\sigma^2(s)\leq C$,
and all the results will be preserved.
\end{remark}

\section{Elements of fractional calculus and existence of the generalized Lebesgue-Stieltjes integrals}\label{sec:integration}

\setcounter{section}{4}
\setcounter{equation}{0}\setcounter{theorem}{0}

\subsection{Elements of fractional calculus}
In this subsection we describe a construction of the path-wise integral following the approach developed by Z\"ahle~\cite{Zahle98,Zahle99,Zahle01}.
We start by introducing the notions of fractional integrals and derivatives.
See~\cite{Samko} for the details on the concept of fractional calculus.

\begin{definition}
Let $f\in L_1(a,b)$.
\emph{The Riemann-Liouville left- and right-sided fractional integrals of order $\alpha>0$} are defined for almost all $x\in(a,b)$ by
\begin{align*}
I^\alpha_{a+} f(x):=\frac1{\Gamma(\alpha)}
\int_a^x(x-y)^{\alpha-1}f(y)\,dy
,\\
I^\alpha_{b-} f(x):=\frac{(-1)^{-\alpha}}{\Gamma(\alpha)}
\int_x^b(y-x)^{\alpha-1}f(y)\,dy,
\end{align*}
respectively, where $(-1)^{-\alpha}=e^{-i\pi\alpha}$, $\Gamma$ denotes the Gamma function.
\end{definition}
Denote by $I^\alpha_{a+}(L_p)$ (resp.\ $I^\alpha_{b-}(L_p)$) the class of functions $f$ that can be presented as
$f=I^\alpha_{a+}\varphi$ (resp.\ $f=I^\alpha_{b-}\varphi$)
for $\varphi\in L_p(a,b)$.

\begin{definition}
For a function $f\colon[a,b]\to \R$
\emph{the Riemann-Liouville left- and right-sided fractional derivatives of order $\alpha$} ($0<\alpha<1$) are defined by
\begin{align*}
\D^\alpha_{a+} f(x):=\ind_{(a,b)}(x)\,\frac1{\Gamma(1-\alpha)}\,\frac{d}{dx}
\int_a^x\frac{f(y)}{(x-y)^{\alpha}}\,dy
,\\
\D^\alpha_{b-}
f(x):=\ind_{(a,b)}(x)\,\frac{(-1)^{1+\alpha}}{\Gamma(1-\alpha)}\,\frac{d}{dx}
\int_x^b\frac{f(y)}{(y-x)^{\alpha}}\,dy.
\end{align*}
\end{definition}
The Riemann-Liouville fractional derivatives admit
the following \emph{Weyl representation}
\begin{align*}
\D^\alpha_{a+} f(x)=\frac1{\Gamma(1-\alpha)}\left(\frac{f(x)}{(x-a)^\alpha}+
\alpha\int_a^x\frac{f(x)-f(y)}{(x-y)^{\alpha+1}}\,dy\right)\ind_{(a,b)}(x),\\
\D^\alpha_{b-}
f(x)=\frac{(-1)^{\alpha}}{\Gamma(1-\alpha)}\left(\frac{f(x)}{(b-x)^\alpha}+
\alpha\int_x^b\frac{f(x)-f(y)}{(y-x)^{\alpha+1}}\,dy\right)\ind_{(a,b)}(x),
\end{align*}
where the convergence of the integrals holds pointwise for a.\,a.\ $x\in(a,b)$ for $p = 1$ and in $L_p(a,b)$ for $p>1$.

\vspace{2mm}
Let $f,g\colon[a,b]\to\R$.
Assume that the limits
\[
f(u+):=\lim_{\delta\downarrow0}f(u+\delta)
\quad\text{and}\quad
g(u-):=\lim_{\delta\downarrow0}f(u-\delta)
\]
exist for $a\le u\le b$.
Denote
\begin{align*}
f_{a+}(x)&=(f(x)-f(a+))\ind_{(a,b)}(x),\\
g_{b-}(x)&=(g(b-)-g(x))\ind_{(a,b)}(x).
\end{align*}

\begin{definition}[\cite{Zahle98}]\label{def:int Stilt}
Assume that
$f_{a+}\in I^\alpha_{a+}(L_p),\  g_{b-}\in
I^{1-\alpha}_{b-}(L_q)$
for some $1/p+1/q\leq1$, $0<\alpha<1$.
\emph{The generalized (fractional) Lebesgue-Stieltjes integral} of $f$ with respect to $g$ is defined by
\begin{equation}\label{eq:int Stilt}
\begin{split}
\int_a^b f(x)\,dg(x):=&(-1)^\alpha\int_a^b
\D^\alpha_{a+}f_{a+}(x)\,\D^{1-\alpha}_{b-}g_{b-}(x)\,dx+\\
&+f(a+)\bigl(g(b-)-g(a+)\bigr).
\end{split}
\end{equation}
\end{definition}
Note  that this definition is independent of the choice of $\alpha$ (\cite[Prop.~2.1]{Zahle98}).
If $\alpha p<1$, then \eqref{eq:int Stilt} can be simplified to
\[
\int_a^b f(x)\,dg(x):=(-1)^\alpha\int_a^b
\D^\alpha_{a+}f(x)\,\D^{1-\alpha}_{b-}g_{b-}(x)\,dx.
\]
In particular, Definition~\ref{def:int Stilt} allows to integrate H\"older continuous functions.
\begin{definition}
Let $0<\lambda\le1$.
A function
$f\colon\R\to\R$
belongs to
$C^\lambda[a,b]$
if there exists a constant $C>0$ such that, for all $s,t\in[a,b]$,
\[
\abs{f(s)-f(t)}\le C\abs{s-t}^\lambda,\quad s,t\in [a,b].
\]
\end{definition}

\begin{proposition}[{\cite[Th~4.2.1]{Zahle98}}]
Let $f\in C^\lambda[a,b]$ and $g\in C^\mu[a,b]$ with $\lambda+\mu>1$.
Then the assumptions of Definition~\ref{def:int Stilt} are satisfied with any $\alpha\in(1-\mu,\lambda)$ and $p=q=\infty$.
Moreover, the generalised Lebesgue-Stieltjes integral $\int_a^b f(x)\,dg(x)$  defined by~\eqref{eq:int Stilt} coincides with the Riemann-Stieltjes integral
\begin{align*}
\{R-S\}\int_a^b f(x)\,dg(x):=\lim_{\abs{\pi}\to0}\sum_i
f(x_i^*)(g(x_{i+1})-g(x_i)),
\end{align*}
where $\pi=\{a=x_0\leq x_0^*\leq x_1\leq \ldots \leq x_{n-1}\leq
x_{n-1}^*\leq x_n=b\}$, and
$\abs{\pi}=\max_i\abs{x_{i+1}-x_i}$.
\end{proposition}

\subsection{Generalised Lebesgue-Stieltjes integral for stochastic processes. }
Consider two real-valued stochastic processes $X=\{X_t, t\in [0,T]\}$ and $Y=\{Y_t, t\in [0,T]\}$. We   say that $X$ and $Y$ are {\it fractionally $\alpha$-connected for some $t\in [0,T]$, and for  some $0<\alpha<1$} if the generalised Lebesgue-Stieltjes integral
\[
\int_0^tX_s\,dY_s:=\int_0^t\left(\D^{\alpha}_{0+}X\right)(s)\left(\D^{1-\alpha}_{t-}Y_{t-}\right)(s)\,ds
\]
exists with probability 1. Since the above integral is defined $\omega$ by $\omega$, it is called a pathwise integral. The next simple result allows us to ``separate" $X$ and $Y$ in the pathwise integral.
\begin{lemma}\label{pr:integr}
Assume that for some $t\in [0,T]$  and for  some $0<\alpha<1$ one of the following conditions hold:
\begin{itemize}
\item[$(i)$] \begin{equation} \int_0^t|\left(\D^{\alpha}_{0+}X\right)(s)|ds<\infty \;\text{a.s. and }\; \sup_{0\leq s\leq t}|\left(\D^{1-\alpha}_{t-}Y_{t-}\right)(s)|<\infty \;\text{a.s.}
    \end{equation}
\item[$(ii)$] \begin{equation} \sup_{0\leq s\leq t}|\left(\D^{\alpha}_{0+}X\right)(s)|<\infty \;\text{a.s. and }\int_0^t|\left(\D^{1-\alpha}_{t-}Y_{t-}\right)(s)|ds<\infty \;\text{a.s.}
    \end{equation}
\item[$(iii)$] for some $p>1$, $q>1$ such that $p^{-1}+q^{-1}=1$
\[
 \int_0^t\left|\left(\D^{\alpha}_{0+}X\right)(s)\right|^qds<\infty
\qquad\text{and}\qquad
 \int_0^t\left|\left(\D^{1-\alpha}_{t-}Y_{t-}\right)(s)\right|^pds<\infty.
\]
\end{itemize}
Then  $X$ and $Y$ are fractionally $\alpha$-connected for this value of $t\in [0,T]$.
\end{lemma}

Taking the above lemma into account, we introduce the classes of stochastic processes
\begin{align*}
\D^+_q(\alpha,T)&:=\set{X=\{X_t, t\in [0,T]\}: \int_0^T|\left(\D^{\alpha}_{0+}X\right)(s)|^qds<\infty \;\text{a.s.}},
\intertext{$1\leq q<\infty$, and}
\D^+_\infty(\alpha,T)&:=\set{X=\{X_t, t\in [0,T]\}: \sup_{0\leq s\leq T}|\left(\D^{\alpha}_{0+}X\right)(s)|<\infty},
\end{align*}
and, correspondingly,
\begin{align*}
\D^{-}_p(\alpha,T)&:=\set{Y=\{Y_t, t\in [0,T]\}:\right.\\ &\quad\left. \int_0^t|\left(\D^{1-\alpha}_{t-}Y_{t-}\right)(s)|^pds<\infty \;\text{a.\,s.}, t\in [0,T]}, 1\leq p<\infty,
\intertext{and}
\D^-_\infty(\alpha,T)&:=\set{Y=\{Y_t, t\in [0,T]\}: \sup_{0\leq s\leq t}|\left(\D^{1-\alpha}_{t-}Y_{t-}\right)(s)|<\infty, t\in [0,T]\!}\!.
\end{align*}
Then it follows that, for the couples ($X\in \D^+_1(\alpha)$, $Y\in \D^-_\infty(\alpha)$),\linebreak ($X\in \D^+_\infty(\alpha)$, $Y\in \D^-_1(\alpha)$), and ($X\in \D^+_q(\alpha)$, $Y\in \D^-_p(\alpha))$, $p>1$, $q>1$, $p^{-1}+q^{-1}=1$, we have that $X$ and $Y$ are fractionally $\alpha$-connected for any  $t\in [0,T]$.

Let the processes $Y\in \D^-_p(\alpha)$ be called {\it appropriate $(p,\alpha)$-integrators}, $p\in[1,+\infty] $ for $X\in \D^+_q(\alpha)$, $q= \frac{p}{p-1}$ (with $\frac{1}{0}=\infty$).

It follows from the a priori estimates of Section 2 that it is natural to formulate conditions on the process
$Y_t=\int_0^tg(t,s)\,dZ_s$
to be appropriate $(p,\alpha)$-integrator in terms of
expectations.
In this connection, we introduce the following classes of processes:
\begin{equation*}\begin{gathered}
\E\D^{-}_p(\alpha,T):=\{Y=\{Y_t, t\in [0,T]\}: \int_0^t\E|\left(\D^{1-\alpha}_{t-}Y_{t-}\right)(s)|^pds<\infty,\\   t\in [0,T]\}\subset\D^-_p(\alpha,T),\;p\geq 1,\end{gathered}\end{equation*}
and
\begin{equation*}\begin{gathered}\E\D^-_\infty(\alpha,T):=\{Y=\{Y_t, t\in [0,T]\}: \E\sup_{0\leq s\leq t}|\left(\D^{1-\alpha}_{t-}Y_{t-}\right)(s)|<\infty,\\ t\in [0,T]\}\subset\D^-_\infty(\alpha,T).\end{gathered}\end{equation*}

\section{General conditions for $Y_\cdot= \int_0^\cdot g(\cdot,s) dZ_s$ to be an appropriate $(p,\alpha)$-integrator}
\label{integration- central}

\setcounter{section}{5}
\setcounter{equation}{0}\setcounter{theorem}{0}

Now we formulate three results supplying the appropriate integrator properties of $Y_t=\int_0^tg(t,s)\,dZ_s, t\in[0,T]$.  Consider the fixed interval $[0,T]$, and let
 $g=g(t,s)\colon\{0\leq s\leq t\leq T\}\rightarrow \R$
be a non-random measurable kernel.

\subsection{The case  $p\in[1,2)$} We immediately formulate the following result.

\begin{theorem}\label{lem7} Let  $p\in[1,2)$,  $ \alpha\in(0,1)$,  $g=g(t,\cdot)\in L_p([0,t])$ for any $t\in[0,T]$,
$a=b=0$, the measure $\pi$ is symmetric with
$\int_\R\abs{x}^p\,\pi(dx)<\infty$, and let the following set of conditions hold:\\
{\bf Assumptions $(D_p)$}
\begin{enumerate}
\item
$\displaystyle
 \int_0^t(t-s)^{p\alpha-p}\left(\int_s^t\abs{g(t,v)}^p\,dv\right)ds<\infty$,
\item
$\displaystyle
\int_0^t(t-s)^{p\alpha-p}\left(\int_0^s\abs{g(t,v)-g(s,v)}^p\,dv\right)ds<\infty$,
\item
$\displaystyle
\int_0^t\int_s^t(u-s)^{p\alpha-2p}\left(\int_s^u\abs{g(u,v)}^p\,dv\right) du\,ds<\infty$,
\item
$\displaystyle
\int_0^t\int_s^t(u-s)^{p\alpha-2p}\left(\int_0^s\abs{g(u,v)-g(s,v)}^p\,dv\right) du\,ds<\infty$.
\end{enumerate}
Then
$Y=\set{Y_t=\int_0^tg(t,s)\,dZ_s, t\in[0,T]}\in \E\D^-_p(\alpha,T)$,
so, $Y$ is an appropriate $(p,\alpha)$-integrator for any $f\in\D^+_q(\alpha, T)$.
\end{theorem}

\begin{proof}
Note that the increment of $Y$ are given by
\begin{equation}\begin{gathered}\label{eq:increm}
Y_t-Y_s=\int_0^tg(t,u)dZ_u-\int_0^sg(s,u)dZ_u\\
=\int_s^tg(t,u)dZ_u+\int_0^s(g(t,u)-g(s,u))dZ_u\quad (s\leq t).
\end{gathered}\end{equation}
Taking the definitions of fractional derivative and of the class $\E\D^-_p(\alpha,T)$ into account, it is sufficient to prove that
\[
\int_0^t\frac{\E\abs{Y_t-Y_s}^p}{(t-s)^{p-\alpha p}}\,ds<\infty
\quad\text{and}\quad
\int_0^t\int_s^t\frac{\E\abs{Y_u-Y_s}^p}{(u-s)^{2p-\alpha p}}\,du\,ds<\infty,
\; t\in[0,T].
\]
According to~\eqref{eq:increm} and \eqref{estimator3},
\begin{align*}
\E\abs{Y_t-Y_s}^p
&\le\E\abs{\int_s^tg(t,u)\,dZ_u}^p+\E\abs{\int_0^s(g(t,u)-g(s,u))\,dZ_u}^p\\
&\le C\!\int_\R |x|^p\pi(dx)\left(\int_s^t\abs{g(t,u)}^pdu
+\!\int_0^s\abs{g(t,u)-g(s,u)}^pdu\right)\\
&\le C\left(\int_s^t\abs{g(t,u)}^p\,du
+\int_0^s\abs{g(t,u)-g(s,u)}^p\,du\right).
\end{align*}
The proof immediately follows.
\end{proof}

Now consider separately the case  $p=2$ because in this case the martingale structure of the process $Y$ plays a crucial role and it is the most simple case for calculations and estimations.

\subsection{The case  $p=2$}
Taking Remark \ref{rem-mart} and Remark \ref{rem kern} into account, we consider a square-integrable c\`{a}dl\`{a}g  martingale $M=\{M_t, \mathcal{F}_t, t\geq 0\}$ with a quadratic  characteristics   $\langle M\rangle$ that is a c\`{a}dl\`{a}g non-decreasing process. Define also a c\`{a}dl\`{a}g non-decreasing measurable function $E_t=\E\langle M\rangle_t.$ The   integral  $\int_0^t g(t,s)\,dM_s$ for any $t>0$ is defined as a stochastic integral with respect to a square-integrable martingale, or, more exactly, since the kernel $g$ is non-random, as a Wiener integral with non-random integrand and a square-integrable martingale as an integrator.
A sufficient condition for its existence is
\[
\E\int_0^tg^2(t,s)\,d\langle M\rangle_s<\infty, \quad t\ge0,
\]
or, equivalently,
\begin{equation}\label{ass1}
 \int_0^tg^2(t,s)\,dE_s<\infty, \quad t\ge0.
\end{equation}
Under assumption  \eqref{ass1}, define the random process
\begin{equation}\label{repr:Y}
Y_t:=\int_0^t g(t,s)\,dM_s,
\quad t\ge0.
\end{equation}

\begin{theorem}\label{pr:integr-Y}
Let $p =2$, $\alpha\in(0,1)$.
Assume that for any $t\in [0,T]$, the following conditions hold.\\
{\bf Assumptions $(D_2)$}
\begin{enumerate}
\item
$\displaystyle
 \int_0^t(t-s)^{2\alpha-2}\int_s^tg^2(t,u)\,dE_u\,ds<\infty$,
\item
$\displaystyle
\int_0^t(t-s)^{2\alpha-2}\int_0^s(g(t,u)-g(s,u))^2\,dE_u\,ds<\infty$,
\item
$\displaystyle
\int_0^t\int_s^t\left(\int_v^t\frac{g(u,v)}{(u-s)^{2-\alpha}}\,du\right)^2 dE_v\,ds<\infty$,
\item
$\displaystyle
 \int_0^t\int_0^s\left(\int_s^t\frac{g(u,v)-g(s,v)}{(u-s)^{2-\alpha}}\,du\right)^2 dE_v\,ds<\infty$.
\end{enumerate}
Then $Y\in \E\D^-_2(\alpha,T),$ so it is an appropriate $(2,\alpha)$-integrator for any $f\in\D^+_2(\alpha, T).$
\end{theorem}

\begin{proof}
By the definition of the fractional derivative,
\begin{multline}\label{eq:int_DY}
\E\int_0^t\left(\left(\D^{1-\alpha}_{t-}Y_{t-}\right)(s)\right)^2ds\\
\le\frac{2}{\Gamma^2(\alpha)}
\Biggl(\E\int_0^t\frac{(Y_t-Y_s)^2}{(t-s)^{2-2\alpha}}\,ds
+(1-\alpha)^2\E\int_0^t\left(\int_s^t\frac{Y_u-Y_s}{(u-s)^{2-\alpha}}\,du\right)^2ds\Biggr).
\end{multline}
Now our goal is to bound from above each of the two terms in the right-hand side of~\eqref{eq:int_DY}.
Similarly to the proof of the previous theorem, the increments of $Y$ are
\begin{align*}
Y_t-Y_s&=\int_0^t g(t,u)\,dM_u-\int_0^s g(s,u)\,dM_u\\
&=\int_s^t g(t,u)\,dM_u+\int_0^s(g(t,u)-g(s,u))\,dM_u \quad (s \leq t).
\end{align*}
Therefore, for the first term we have the following upper bound
\begin{equation}\label{eq:1st_term}
\begin{split}
&\E\int_0^t\frac{(Y_t-Y_s)^2}{(t-s)^{2-2\alpha}}\,ds\\
&\le2\E\int_0^t(t-s)^{2\alpha-2}
\left(\left(\int_s^t g(t,u)\,dM_u\right)^2\right.\\
&\qquad+\left.\left(\int_0^s(g(t,u)-g(s,u))\,dM_u\right)^2\right)ds\\
&=2\int_0^t(t-s)^{2\alpha-2}
\left(\E\int_s^t g^2(t,u)\,d\langle M\rangle_u
\right.\\
&\qquad+\left.\E\int_0^s(g(t,u)-g(s,u))^2\,d\langle M\rangle_u\right)ds\\&=2\int_0^t(t-s)^{2\alpha-2}
\left(\int_s^t g^2(t,u)\,dE_u
+\int_0^s(g(t,u)-g(s,u))^2\,dE_u\right)ds.
\end{split}
\end{equation}
Hence, the first expectation in~\eqref{eq:int_DY} is finite.
The second summand in the right-hand side of ~\eqref{eq:int_DY} can be bounded as follows:
\begin{align*}
\E&\int_0^t\left(\int_s^t\frac{Y_u-Y_s}{(u-s)^{2-\alpha}}\,du\right)^2ds\\
&=\E\int_0^t\left(\int_s^t(u-s)^{\alpha-2}
\left(\int_s^ug(u,v)\,dM_v\right.\right.\\
&\qquad+\left.\left.\int_0^s(g(u,v)-g(s,v))\,dM_v\right)\,du\right)^2ds\\
&\le2\E\int_0^t\left(\left(\int_s^t\int_s^u\frac{g(u,v)}{(u-s)^{2-\alpha}}\,dM_v\,du\right)^2\right.\\
&\qquad+\left.\left(\int_s^t\int_0^s\frac{g(u,v)-g(s,v)}{(u-s)^{2-\alpha}}\,dM_v\,du\right)^2\right)ds\\
&=2\int_0^t\left(\E\left(\int_s^t\int_v^t\frac{g(u,v)}{(u-s)^{2-\alpha}}\,du\,dM_v\right)^2\right.\\
&\qquad+\left.\E\left(\int_0^s\int_s^t\frac{g(u,v)-g(s,v)}{(u-s)^{2-\alpha}}\,du\,dM_v\right)^2\right)ds\\
&=2\int_0^t\left(\E\int_s^t\left(\int_v^t\frac{g(u,v)}{(u-s)^{2-\alpha}}\,du\right)^2d\langle M\rangle_v\right.\\
&\qquad+\left.\E\int_0^s\left(\int_s^t\frac{g(u,v)-g(s,v)}{(u-s)^{2-\alpha}}\,du\right)^2d\langle M\rangle_v\right)ds\\
&=2\int_0^t\left(\int_s^t\left(\int_v^t\frac{g(u,v)}{(u-s)^{2-\alpha}}\,du\right)^2dE_v\right.\\
&\qquad+\left.\int_0^s\left(\int_s^t\frac{g(u,v)-g(s,v)}{(u-s)^{2-\alpha}}\,du\right)^2dE_v\right)ds
<\infty.
\end{align*}
\end{proof}

\vspace{3mm}
\subsection{The case  $2<p<\infty$} Taking inequality \eqref{ineq-p} into account, we can formulate corresponding  result similarly to Theorem \ref{lem7}.
Since the proof follows the same steps, it is omitted.

\begin{theorem}\label{lem8} Let  $ p\in(2,\infty), \alpha\in(0,1)$,    $g=g(t,\cdot)\in L_p([0,t])$ for any $t\in[0,T]$,
$a=0$, the measure $\pi$ is symmetric with
$\int_\R\abs{x}^p\,\pi(dx)<\infty$,
  and let additionally condition $(D_p)$ hold.
Then
$$Y=\set{Y_t=\int_0^tg(t,s)\,dZ_s, t\in[0,T]}\in \E\D^-_p(\alpha,T),$$
so, it is an appropriate $(p,\alpha)$-integrator for any $f\in\D^+_q(\alpha, T)$.
\end{theorem}

\subsection{The case  $p=\infty$} As in the case $p=2$, consider  a square-integrable    martingale $M=\{M_t, \mathcal{F}_t, t\geq 0\}$ with a  quadratic  characteristics   $\langle M\rangle$. In order to give the conditions for $Y\in\E \D^-_\infty(\alpha,T)$, a H\"older continuity of rather high order is required.
Therefore, we assume that $M$ and consequently  $\langle M\rangle$ are continuous processes, and that  $E_t=\E\langle M\rangle_t $ is a continuous function.
Remark immediately that this is not the case for subordinate Wiener process.

\begin{theorem}\label{pr:integr-Y-3}
Let $\alpha\in(0,1)$, $M$ is a square-integrable continuous martingale with quadratic characteristics $\langle M\rangle_t=\int_0^t m_sds $, where $\E m_s\leq C$.
Assume that, for some $\varrho\ge1$, $\beta>\frac1\varrho+1-\alpha$, the following condition holds:\\
{\bf Assumptions $(D_\infty)$}
\begin{enumerate}
\item
$\displaystyle
\int_0^T\!\!\int_0^T\abs{y-x}^{-\beta \varrho-1}
\left(\int_x^yg^2(y,u)\,du\right)^{\varrho/2}dx\,dy<\infty
$
\item
$\displaystyle
\int_0^T\!\!\int_0^T\abs{y-x}^{-\beta \varrho-1}
\left(\int_0^x(g(y,u)-g(x,u))^2\,du\right)^{\varrho/2}dx\,dy<\infty.
$
\end{enumerate}
Then $Y\in\E \D^-_\infty(\alpha,T),$ so it is an appropriate $(\infty,\alpha)$-integrator for any $f\in\D^+_0(\alpha, T).$
\end{theorem}

\begin{proof}
The proof follows the scheme from~\cite[Lemma~7.5]{NR}.
According to the Garsia-Rodemich-Rumsey inequality from \cite{GRR}, for any continuous function $f\colon[0,T]\to\R$ and any $\varrho\ge1$, $\beta>\frac1\varrho$ and $0\le s\le t\le T$, we have
\[
\abs{f(t)-f(s)}^\varrho\le C_{\beta,\varrho}\abs{t-s}^{\beta \varrho-1}
\int_0^T\!\!\int_0^T\frac{\abs{f(x)-f(y)}^\varrho}{\abs{x-y}^{\beta \varrho+1}}\,dx\,dy.
\]
Consider the increment of $Y$:
\[
\abs{Y_t-Y_s}^\varrho\le C\abs{t-s}^{\beta \varrho-1}
\int_0^T\!\!\int_0^T\frac{\abs{Y_x-Y_y}^\varrho}{\abs{x-y}^{\beta \varrho+1}}\,dx\,dy \quad (s \leq t).
\]
Taking the continuity of $Y$ into account, we can apply Burkholder's inequality, and get
\begin{equation}\label{eq:3}
\begin{split}
\E\abs{Y_x-Y_y}^\varrho&\le C\E\abs{\int_x^yg(y,u)\,dM_u}^\varrho
+ C\E\abs{\int_0^x(g(y,u)-g(x,u))\,dM_u}^\varrho\\
&\le C\left(\int_x^yg^2(y,u)\,du\right)^{\varrho/2}
+ C\left(\int_0^x(g(y,u)-g(x,u))^2\,du\right)^{\varrho/2}.
\end{split}
\end{equation}
Then
\begin{equation}\label{eq:4}
\abs{Y_t-Y_s}\le C\abs{t-s}^{\beta-1/\varrho}\xi,
\end{equation}
where
\[
\xi=\left(\int_0^T\!\!\int_0^T\frac{\abs{Y_x-Y_y}^\varrho}{\abs{x-y}^{\beta \varrho+1}}\,dx\,dy\right)^{1/\varrho}.
\]
By the upper bound ~\eqref{eq:3} and condition $(D_\infty)$,
$\E\abs{\xi}^\varrho<\infty$.
By the definition of the fractional derivative,
\[
\abs{\D^{1-\alpha}_{t-}Y_{t-}(s)}\le\frac{1}{\Gamma(\alpha)}\left(\frac{\abs{Y_t-Y_s}}{(t-s)^{1-\alpha}}+
(1-\alpha)\int_s^t\frac{\abs{Y_s-Y_v}}{(v-s)^{2-\alpha}}\,dv\right).
\]
Combining this with~\eqref{eq:4}, and using the inequality $\beta>\frac1\varrho+1-\alpha$, we complete the proof.
\end{proof}

\begin{corollary}\label{cor:1}
Let the function $g$ satisfy the following conditions:
\begin{enumerate}
\item
$\abs{g(t,s)}\le C$, $t,s\in[0,T]$;
\item
$\abs{g(t,s)-g(v,s)}\le C\abs{t-v}^{1/2}$, $t,s,v\in[0,T]$.
\end{enumerate}
Then the assumptions of Theorem~\ref{pr:integr-Y-3} are satisfied for any $\alpha>1/2$, if we choose $\varrho\ge\frac{2}{2\alpha-1}$ and $\frac1\varrho+1-\alpha<\beta<\frac12$.
\end{corollary}

\subsection{Conditions for $\int_\cdot = \int_0^\cdot g(\cdot,s) dW^L_s$ to be an appropriate $(p,\alpha)$-integrator}

As an example in line with Section 3, we formulate the results supplying the appropriate integrator properties of $Y_t=\int_0^tg(t,s)\,dW^L_s, t\in[0,T]$ for $Z=W^L$ a subordinated Wiener process.
Consider the fixed interval $[0,T]$, and let
 $g=g(t,s)\colon \{0\leq s \leq t \leq T\}  \rightarrow \R$
be a non-random measurable kernel.
\begin{theorem}\label{theor-Wiener-1} Let  $  p\in[1,\infty), \alpha \in(0,1)$,  $g=g(t,\cdot)\in L_p([0,t])$ for any $t\in[0,T]$, $a=0$ in the case when $1\leq p\leq 2$, and $\int_0^\infty s^{\frac{p}{2}}d\nu_s<\infty.$
Also, let   conditions $(D_p)$ hold.
Then
$$Y=\set{Y_t=\int_0^tg(t,s)\,dW^L_s, t\in[0,T]}\in \E\D^-_p(\alpha,T),$$
so, $Y$ is an appropriate $(p,\alpha)$-integrator for any $f\in\D^+_q(\alpha, T)$.
\end{theorem}

\subsection{Examples of appropriate $(p,\alpha)$-integrators}

\begin{example}\label{ex:1}
Assume that on interval $[0,T]$ two following properties hold:
\begin{enumerate}[(i)]
\item
$\E\langle M\rangle_t=\int_0^t\sigma^2(s)\,ds$,
where
$\abs{\sigma(s)}\le \sigma$, where $\sigma>0$ is some constant;
\item
$g(t,s)=g(j(\cdot), t,s)=c_Hs^{\frac12-H}\int_s^tu^{H-\frac12}(u-s)^{H-3/2}j(u)\,du$,
where
  $H\in\left(\frac12,1\right)$, $j$ is a measurable bounded function, $|j(u)|\leq G$, where $G>0$ is some constant.
\end{enumerate}
Then the process $Y$ from \eqref{repr:Y} is an appropriate $(2,\alpha)$-integrator for any $1-H<\alpha<1$, and in addition, has a.\,s.\ $\gamma$-H\"older  trajectories for any  $0<\gamma<H-1/2$.
\end{example}

\begin{proof}
From now on, we denote $C$ different constants whose value is not here important. Start  with H\"older continuity. According to Kolmogorov theorem, it is sufficient to prove that
\begin{equation}\label{eq:inc_var}
\E(Y_t-Y_s)^2\le C(t-s)^{2H}=C(t-s)^{1+2(H-1/2)}.
\end{equation}
To establish \eqref{eq:inc_var}, we note that, similarly to~\eqref{eq:1st_term}, for $s \leq t$,
\begin{align*}
\E(Y_t-Y_s)^2&\le2\left(\E\int_s^t g^2(j(\cdot), t,u)\,d\langle M\rangle_u\right.\\
&\quad\left.+\E\int_0^s(g(j(\cdot), t,u)-g(j(\cdot), s,u))^2\,d\langle M\rangle_u\right)\\
&=2\left(\E\int_s^t g^2(j(\cdot), t,u)\sigma^2(u)\,du\right.\\
&\quad\left.
+\E\int_0^s(g(j(\cdot), t,u)-g(j(\cdot), s,u))^2\sigma^2(u)\,du\right)\\
&\le2\sigma^2(I_1+I_2),
\end{align*}
where
\[
I_1=\int_s^t g^2(j(\cdot), t,u)\,du,\\
\qquad
I_2=\int_0^s(g(j(\cdot), t,u)-g(j(\cdot), s,u))^2\,du.
\]
Consider $I_1$.
\begin{align*}
I_1&=c_H^2\int_s^tu^{1-2H}\left(\int_u^tv^{H-\frac12}(v-u)^{H-\frac32}j(v)dv\right)^2du\\
&\leq c_H^2 G^2\int_s^t\int_u^t\int_u^tf(u,v,z)\,dz\,dv\,du,
\end{align*}
where
\[
f(u,v,z)= u^{1-2H}v^{H-\frac12}(v-u)^{H-\frac32}z^{H-\frac12}(z-u)^{H-\frac32}.
\]
Changing the order of integration and using the equality
$f(u,v,z)=f(u,z,v)$, we get
\begin{align*}
 &\int_s^t\int_s^v\int_s^zf(u,v,z)\,dz\,du\,dv\\
&=\int_s^t\int_s^v\int_s^zf(u,v,z)\,du\,dz\,dv
+\int_s^t\int_v^t\int_s^vf(u,v,z)\,du\,dz\,dv\\
&=\int_s^t\int_s^v\int_s^zf(u,v,z)\,du\,dz\,dv
+\int_s^t\int_s^z\int_s^vf(u,v,z)\,du\,dv\,dz\\
&=2\int_s^t\int_s^v\int_s^zf(u,v,z)\,du\,dz\,dv.
\end{align*}
Similarly $I_2$ can be rewritten  as follows
\begin{align*}
I_2&=c_H^2\int_0^su^{1-2H}\left(\int_s^tv^{H-\frac12}(v-u)^{H-\frac32}g(v)dv\right)^2du\\
&\leq c_H^2 G^2\int_0^s\int_s^t\int_s^tf(u,v,z)\,dz\,dv\,du\\
&=c_H^2 G^2\int_s^t\int_s^t\int_0^sf(u,v,z)\,du\,dz\,dv\\
&=2c_H^2 G^2\int_s^t\int_s^v\int_0^sf(u,v,z)\,du\,dz\,dv.
\end{align*}
Summarizing, we get that
\begin{align*}
I_1+I_2&\leq C\int_s^t\int_s^v\int_0^zf(u,v,z)\,du\,dz\,dv\\
&=C\int_s^t\int_s^vv^{H-\frac12}z^{H-\frac12}
\int_0^zu^{1-2H}(v-u)^{H-\frac32}(z-u)^{H-\frac32}\,du\,dz\,dv.
\end{align*}
Using~\cite[Lemma~2.2(i)]{NVV} we can calculate the inner integral:
\begin{equation}\label{eq:frac_integral}
\int_0^zu^{1-2H}(z-u)^{H-\frac32}
(v-u)^{H-\frac32}\,du
=Cv^{\frac12-H}z^{\frac12-H}(v-z)^{2H-2}.
\end{equation}
Then
\begin{equation}\label{eq:I1+I2}
\begin{split}
I_1+I_2
&\leq C\int_s^t\int_s^v(v-z)^{2H-2}\,dv =C(t-s)^{2H}.
\end{split}
\end{equation}
Thus, \eqref{eq:inc_var} is proved.
Now it remains to   check the conditions of Theorem~\ref{pr:integr-Y}.
Since
$d\E\langle M\rangle_s=\sigma^2(s)\,ds$ and
$\abs{\sigma(s)}\le C$, it is very easy to understand that
we need to show   the existence of the following four integrals:
\begin{align*}
J_1&=\int_0^t(t-s)^{2\alpha-2}\int_s^tg^2(1, t,u)\,du\,ds,\\
J_2&=\int_0^t(t-s)^{2\alpha-2}\int_0^s(g(1, t,u)-g(1, s,u))^2\,du\,ds,\\
J_3&=\int_0^t\int_s^t\left(\int_v^t\frac{g(1, u,v)}{(u-s)^{2-\alpha}}\,du\right)^2 dv\,ds,\\
J_4&=\int_0^t\int_0^s\left(\int_s^t\frac{g(1, u,v)-g(1,s,v)}{(u-s)^{2-\alpha}}\,du\right)^2 dv\,ds,
\end{align*}
where we replaced $j(\cdot)$ with the constant function identically equal to 1.
Using~\eqref{eq:I1+I2}, we can bound the first two integrals by
\begin{align*}
J_1+J_2&=\int_0^t(t-s)^{2\alpha-2}\left(I_1+I_2\right)\,ds
=\int_0^t(t-s)^{2H+2\alpha-2}\,ds<\infty.
\end{align*}
Consider $ J_3$.
\begin{align*}
J_3&=c_H^2\int_0^t\int_s^tv^{1-2H}\left(\int_v^t(u-s)^{\alpha-2}\int_v^uz^{H-\frac12}(z-v)^{H-\frac32}\,dz\,du\right)^2 dv\,ds\\
&\le c_H^2t^{2H-1}\int_0^t\int_s^tv^{1-2H}\left(\int_v^t(u-s)^{\alpha-2}\int_v^u(z-v)^{H-\frac32}\,dz\,du\right)^2 dv\,ds\\
&=C t^{2H-1}
\int_0^t\int_s^tv^{1-2H}\left(\int_v^t(u-s)^{\alpha-2}(u-v)^{H-\frac12}\,du\right)^2 dv\,ds\\
&\le C t^{2H-1}
\int_0^t\int_s^tv^{1-2H}\left(\int_v^t(u-s)^{H+\alpha-5/2}\,du\right)^2 dv\,ds.
\end{align*}
The convergence of the integral for some $\alpha=\alpha_0\in(0,1)$ implies its convergence for all $\alpha\in[\alpha_0,1)$.
Therefore, we can assume without loss of generality that
$\alpha<\frac32-H$.
Then
\begin{align*}
J_3&\le C t^{2H-1}
\int_0^t\int_s^tv^{1-2H}\left((t-s)^{H+\alpha-\frac32}-(v-s)^{H+\alpha-\frac32}\right)^2 dv\,ds\\
&\le C t^{2H-1}
\int_0^t\int_s^tv^{1-2H}(v-s)^{2H+2\alpha-3}\,dv\,ds.
\end{align*}
By changing the order of integration, we get
\[
J_3\le C t^{2H-1}
\int_0^tv^{1-2H}\int_0^v(v-s)^{2H+2\alpha-3}\,ds\,dv=C t^{2H-1}
\int_0^tv^{2\alpha-1}\,dv<\infty.
\]
Consider $ J_4$.
\begin{align*}
J_4&=c_H^2\int_0^t\int_0^sv^{1-2H}\left(\int_s^t(u-s)^{\alpha-2}
\int_u^sz^{H-\frac12}(z-v)^{H-\frac32}\,dz\,du\right)^2 dv\,ds\\
&\le c_H^2t^{2H-1}\int_0^t\int_0^sv^{1-2H}\left(\int_s^t(u-s)^{\alpha-2}
\int_u^s(z-v)^{H-\frac32}\,dz\,du\right)^2 dv\,ds\\
&= c_H^2t^{2H-1}\int_0^t\int_0^sv^{1-2H}\left(\int_s^t\int_z^t(u-s)^{\alpha-2}
(z-v)^{H-\frac32}\,du\,dz\right)^2 dv\,ds\\
&\le C t^{2H-1}\int_0^t\int_0^sv^{1-2H}\left(\int_s^t(z-s)^{\alpha-1}
(z-v)^{H-\frac32}\,dz\right)^2 dv\,ds\\
&= C t^{2H-1}\int_0^t\int_0^sv^{1-2H}\int_s^t(z-s)^{\alpha-1}
(z-v)^{H-\frac32}\,dz\\
&\qquad\times\int_s^t(y-s)^{\alpha-1}
(y-v)^{H-\frac32}\,dy\,dv\,ds\\
&= C t^{2H-1}\int_0^t\int_s^t\int_s^t(z-s)^{\alpha-1}(y-s)^{\alpha-1}\\
&\qquad\times\int_0^sv^{1-2H}(z-v)^{H-\frac32}
(y-v)^{H-\frac32}\,dv\,dz\,dy\,ds\\
&= C t^{2H-1}\int_0^t\int_s^t\int_s^y(z-s)^{\alpha-1}(y-s)^{\alpha-1}\\
&\qquad\times\int_0^sv^{1-2H}(z-v)^{H-\frac32}
(y-v)^{H-\frac32}\,dv\,dz\,dy\,ds\\
&\le C t^{2H-1}\int_0^t\int_s^t\int_s^y(z-s)^{\alpha-1}(y-s)^{\alpha-1}\\
&\qquad\times\int_0^zv^{1-2H}(z-v)^{H-\frac32}
(y-v)^{H-\frac32}\,dv\,dz\,dy\,ds.
\end{align*}
By~\eqref{eq:frac_integral}, we have
\begin{align*}
J_4&\le C t^{2H-1}\int_0^t\int_s^t\int_s^y(z-s)^{\alpha-1}(y-s)^{\alpha-1}z^{\frac12-H}y^{\frac12-H}(y-z)^{2H-2}\,dz\,dy\,ds\\
&\le C t^{2H-1}
\int_0^ts^{1-2H}\int_s^t(y-s)^{\alpha-1}\int_s^y(z-s)^{\alpha-1}(y-z)^{2H-2}\,dz\,dy\,ds\\
&=C t^{2H-1}
\int_0^ts^{1-2H}\int_s^t(y-s)^{2H+2\alpha-3}\,dy\,ds\\
&=C t^{2H-1}
\int_0^ts^{1-2H}(t-s)^{2H+2\alpha-2}\,ds<\infty.
\end{align*}
This concludes the proof.
\end{proof}

\begin{example}
Set i$c_H=\left(\frac{H(2H-1)}{B(2-2H,H-\frac12)}\right)^{\frac12}$,
and $j(u)\equiv1$ n Example~\ref{ex:1}.
Then $g(1,t,s)$ is the Molchan-Golosov kernel.
If $M$ is a Wiener process then the process
$Y_t=\int_0^tg(1,t,s)\,ds$ is the fractional Brownian motion, see~\cite{NVV}.
Note, that in this case trajectories of $Y$ are a.\,s.\ $\gamma$-H\"older  for any  $0<\gamma<H$.
The pathwise generalized Lebesgue-Stiltjes integrals with respect to fractional Brownian motion was studied in~\cite{NR}.
If $M$ is a L\'evy process without Gaussian component, then $Y$ is fLpMG, introduced in~\cite{Tikanmaki}.
\end{example}

\begin{example}
It is very easy to create examples of processes from $  \E\D^-_1(\alpha,T) $ and $ \D^-_\infty(\alpha,T)$. Indeed we can take the same kernel $g(1,t,s)$ and consider any $W^L$ satisfying condition $(A)$ with $a=0$ to get that $ Y\in \E\D^-_0(\alpha,T) $. Moreover, with the same kernel and $M=W$ we get $ Y\in\D^-_\infty(\alpha,T)$, as it immediately follows from \cite{NR}.
\end{example}

\section*{Acknowledgements}

The authors thank the EU project Ukrainian Mathematicians for Life Sciences for providing the framework for this research. Giulia Di Nunno acknowledges financial support from the Norwegian Research Council Project 239019 FINEWSTOCH.


\end{document}